\documentclass[11pt]{amsart}
\usepackage{pifont}
\usepackage{booktabs, siunitx,  txfonts}

\usepackage{amssymb, amsfonts, amsmath, amsthm}
 \usepackage{xcolor, colortbl}
 \usepackage{tabularx}
 \usepackage{graphicx}
 \usepackage[scriptsize,up]{caption}
\usepackage{wrapfig}
\usepackage{pgf,tikz}
\usetikzlibrary{decorations, decorations.pathmorphing}
\usetikzlibrary {shapes}

\usepackage{setspace}
\usepackage{ifpdf}
\input xy
\xyoption{all}
\ifpdf
\usepackage{hyperref}
\else
\usepackage[hypertex]{hyperref}
\fi

\theoremstyle{plain}
\newtheorem{thm}{Theorem}[section]
\newtheorem*{main}{Main Theorem}
\newtheorem{lem}[thm]{Lemma}

\newtheorem{qu}[thm]{Question}

\numberwithin{equation}{section}
\theoremstyle{definition}

\theoremstyle{remark}

\newcommand{\un}{\uparrow} 
\newcommand{\de}{\downarrow} 
\newcommand\nbd{\nobreakdash-\hspace{0pt}}
\newcommand\oh{{\mathbf 0}}
\newcommand\ba{{\mathbf a}}
\newcommand\bb{{\mathbf b}}
\newcommand\bx{{\mathbf x}}
\newcommand\by{{\mathbf y}}
\newcommand\cP{{\mathcal P}}
\newcommand\cQ{{\mathcal Q}}
\newcommand\cR{{\mathcal R}}
\newcommand\cS{{\mathcal S}}

\begin{document}

\title{On the existence of a strong minimal pair}

\author[George Barmpalias]{George Barmpalias}
\address{{\bf George Barmpalias:}
(1) State Key Lab of Computer Science,
Institute of Software,
Chinese Academy of Sciences,
Beijing 100190,
China and (2) School of Mathematics, Statistics and Operations Research,
Victoria University, Wellington, New Zealand}
\email{barmpalias@gmail.com}
\urladdr{\href{http://www.barmpalias.net}{http://www.barmpalias.net}}

\author[Mingzhong Cai]{Mingzhong Cai}
\address{{\bf Mingzhong Cai:} Department of Mathematics,
Dartmouth College, Hanover, NH 03755, USA}
\email{Mingzhong.Cai@dartmouth.edu}
\urladdr{\href{http://math.dartmouth.edu/~cai}%
{http://math.dartmouth.edu/~cai}}

\author[Steffen Lempp]{Steffen Lempp}
\address{{\bf Steffen Lempp:} Department of Mathematics,
University of Wisconsin, Madison, WI 53706, USA}
\email{lempp@math.wisc.edu}
\urladdr{\href{http://www.math.wisc.edu/~lempp}%
{http://www.math.wisc.edu/~lempp}}

\author[Theodore A. Slaman]{Theodore A. Slaman}
\address{{\bf Theodore A. Slaman:} Department of Mathematics,
University of California, Berkeley, CA 94720, USA}
\email{slaman@math.berkeley.edu}
\urladdr{\href{http://www.math.berkeley.edu/~slaman}%
{http://www.math.berkeley.edu/~slaman}}

\thanks{This research was partially carried out while the first, third and
fourth authors were visiting fellows at the Isaac Newton Institute for the
Mathematical Sciences in the program ``Semantics \&\ Syntax''. The first
author's research was supported by the {\em Research fund for international
young scientists} numbers 613501-10236 and 613501-10535 from the National
Natural Science Foundation of China, and an {\em International Young
Scientist Fellowship} number 2010-Y2GB03 from the Chinese Academy of
Sciences. The second author's research was partly supported by an AMS Simons
Travel Grant and NSF Grant DMS-1266214. The third author's research was
partly supported by AMS-Simons Foundation Collaboration Grant 209087. The
fourth author's research was supported by the National Science Foundation,
USA, under Grant No. DMS-1001551 and by the Simons Foundation.}

\date{\today}

\begin{abstract}
We show that there is a strong minimal pair in the computably enumerable
Turing degrees, i.e., a pair of nonzero c.e.\ degrees~$\ba$ and~$\bb$ such
that $\ba \cap \bb = \oh$ and for any nonzero c.e.\ degree $\bx \le \ba$,
$\bb \cup \bx \ge \ba$.
\end{abstract}

\maketitle

\section{Introduction}

Much of the work on the degree structure of the computably enumerable (c.e.)\
Turing degrees has focused on studying its finite substructures and how they
can be extended to larger substructures. There are several reasons for this:
The partial order of the c.e.\ degrees is a very complicated algebraic
structure, with an undecidable first-order theory, by Harrington and
Shelah~\cite{HS82}. So, on the one hand, as in classical algebra, a
complicated structure is often best understood by studying its finite
substructures. On the other hand, the existential fragment of the first-order
theory of this degree structure (in the language of partial ordering~$<$,
least element~0 and greatest element~1) is known to be decidable by
Sacks~\cite{Sa63}, whereas by Lempp, Nies and Slaman~\cite{LNS98}, the
$\exists\forall\exists$\nbd theory of this structure (in the language of
partial ordering only) is undecidable. However, the decidability of the
$\forall\exists$\nbd theory of this structure has been an open question for a
long time; and it is this question which can be rephrased in purely algebraic
terms as a question about finite substructures:

\begin{qu}[Extendibility Question]\label{qu:extendqu}
Let~$\cP$ and~$\cQ_i$ (with $i < n$) be finite posets such that for all $i <
n$, $\cP \subseteq \cQ_i$. Under what conditions on~$\cP$ and the~$\cQ_i$ can
any embedding of~$\cP$ into the c.e.\ Turing degrees be extended to an
embedding of~$\cQ_i$ into the c.e.\ Turing degrees for some~$i$ (which may
depend on the embedding of~$\cP$)?
\end{qu}

Call a partially ordered set~$\cP$ \emph{bounded} if it contains
distinguished least and great elements~0 and~1, respectively. We can now
formulate the following modified

\begin{qu}[Extendibility Question with~0 and~1]\label{qu:extendqu01}
Let~$\cP$ and~$\cQ_i$ (with $i < n$) be finite bounded posets such that for
all $i < n$, $\cP \subseteq \cQ_i$. Under what conditions on~$\cP$ and
the~$\cQ_i$ can any embedding of~$\cP$ into the c.e.\ Turing degrees
(preserving~0 and~1) be extended to an embedding of~$\cQ_i$ into the c.e.\
Turing degrees (preserving~0 and~1) for some~$i$?
\end{qu}

The answer for $n=1$ to Question~\ref{qu:extendqu01} was given by the
following

\begin{thm}[Slaman, Soare~\cite{SS99}]\label{th:slsoare}
Uniformly in finite bounded posets~$\cP$ and~$\cQ$, there is an effective
procedure to decide whether any embedding of~$\cP$ into the c.e.\ Turing
degrees (preserving~0 and~1) be extended to an embedding of~$\cQ$ into the
c.e.\ Turing degrees (preserving~0 and~1).
\end{thm}

This result of Slaman and Soare built on a long line of research into the
algebraic structure of the c.e.\ degrees, starting with the Sacks Splitting
and Density Theorems~\cite{Sa63b, Sa64} and the proof of the existence of a
minimal pair of c.e.\ degrees by Lachlan~\cite{La66} and Yates~\cite{Ya66}.

In their proof in~\cite{SS99}, Slaman and Soare identify two basic obstacles
to extending an embedding. The first of these is lattice-theoretic: The c.e.\
degrees form an upper semilattice in which the meet of some but not all pairs
of degrees exists. In fact, the major hurdle toward deciding the
$\forall\exists$\nbd theory of this structure has been the long-standing
lattice embeddings problem, asking for an (effective) characterization of
those finite lattices which can be embedded into the c.e.\ Turing degrees.
(Note that the lattice embeddings problem can be phrased as a subproblem of
the Extendibility Question by making all the~$\cQ_i$ one-point extensions
of~$\cP$, each testing the preservation of a particular meet or join in the
lattice embedding. Lerman~\cite{Le00} gave a noneffective (indeed a
$\Pi^0_2$-)condition for lattice embeddability; a more recent survey is
Lempp, Lerman and Solomon~\cite{LLS06}.)

The other basic obstacle to extending an embedding identified by Slaman and
Soare is a phenomenon sometimes called ``saturation''; a minimal example of
it is given by setting $\cP = \{0,a,b,1\}$ (with incomparable $a,b$) and $\cQ
= \cP \cup \{x,z\}$ (with $0<x<a,z$ and $b<z<1$ but $x \nleq b$ and $a \nleq
z$). In the general case, there may be a number of such elements $x \in
\cQ-\cP$, and for each~$x$ there will be a non-empty set $Z(x) \subseteq
\cQ-\cP$ of such~$z$.

An early example of a specific instance of an answer to the Extendibility
Question~\ref{qu:extendqu01} for $n>1$ was given by Lachlan's Nondiamond
Theorem~\cite{La66}: No minimal pair of c.e.\ degrees can cup to~$\oh'$. (For
this, we set $\cP = \{0,a,b,1\}$ (with incomparable $a,b$), $\cQ_0 = \cP \cup
\{x\}$ (with $0<x<a,b$), and $\cQ_1 = \cP \cup \{y\}$ (with $a,b<y<1$).) So
this is an instance of two lattice-theoretic obstructions which cannot be
overcome individually, but can be overcome in combination.

The main theorem of this paper provides an example where a lattice-theoretic
obstruction and a ``saturation'' obstruction cannot each be overcome either
individually or in combination:

\begin{main}
There is a \emph{strong minimal pair} in the c.e.\ Turing degrees, i.e.,
there are nonzero c.e.\ degrees~$\ba$ and~$\bb$ such that $\ba \cap \bb =
\oh$ and for any nonzero c.e.\ degree $\bx \le \ba$, $\bb \cup \bx \ge \ba$.
\end{main}

Note that this is an instance of the Extendibility
Question~\ref{qu:extendqu01} by setting $\cP = \{0,a,b,1\}$ (with
incomparable $a,b$), $\cQ_0 = \cP \cup \{x\}$ (with $0<x<a,b$) and $\cQ_1 =
\cP \cup \{x,z\}$ (with $0<x<a,z$ and $b<z<1$ but $x \nleq b$ and $a \nleq
z$).

We should mention here that our Main Theorem has a long and twisted history.
It was discussed and claimed, in both directions, by a number of researchers
over the past 25 years. The only published proof is in Lerman's
monograph~\cite{Le10}, where he attributes the theorem to Slaman (also see
the review by Barmpalias~\cite{Ba11}). However (per personal communication
with Lerman), the proof published by Lerman~\cite{Le10} has a gap, which is
filled by a feature which we introduce in our proof here.

We would like to state here the following related question, which we leave
open:

\begin{qu}\label{qu:2smp}
Is there a ``two-sided'' strong minimal pair; i.e., are there nonzero c.e.\
degrees~$\ba$ and~$\bb$ such that $\ba \cap \bb = \oh$, for any nonzero c.e.\
degree $\bx \le \ba$, $\bb \cup \bx \ge \ba$, and for any nonzero c.e.\
degree $\by \le \bb$, $\ba \cup \by \ge \bb$?
\end{qu}

This is, of course, an instance of the Extendibility
Question~\ref{qu:extendqu01} (with $n=3$, combining one lattice-theoretic and
two ``saturation'' obstructions, namely, setting $\cP = \{0,a,b,1\}$ (with
incomparable $a,b$), $\cQ_0 = \cP \cup \{w\}$ (with $0<w<a,b$), $\cQ_1 = \cP
\cup \{x,z\}$ (with $0<x<a,z$ and $b<z<1$ but $x \nleq b$ and $a \nleq z$),
and $\cQ_2 = \cP \cup \{x',z'\}$ (with $0<x'<b,z'$ and $a<z'<1$ but $x' \nleq
a$ and $b \nleq z'$). We remark here that our Question~\ref{qu:2smp} has a
negative answer if we also require the join of (the images of)~$a$ and~$b$ to
be ``branching'' (i.e., meet-reducible); i.e., any embedding of $\cP =
\{0,a,b,c,d,e,1\}$ (with incomparable $a,b$, incomparable $d,e$, and
$a,b<c<d,e$) extends to an embedding of $\cQ_0 = \cP \cup \{w\}$ (with
$0<w<a,b$), $\cQ_1 = \cP \cup \{x,z\}$ (with $0<x<a,z$ and $b<z<1$ but $x
\nleq b$ and $a \nleq z$), $\cQ_2 = \cP \cup \{x',z'\}$ (with $0<x'<b,z'$ and
$a<z'<1$ but $x' \nleq a$ and $b \nleq z'$), $\cQ_3 = \cP \cup \{y\}$ (with
$a,b<y<c$), or $\cQ_4 = \cP \cup \{y'\}$ (with $c<y'<d,e$). This last result
was observed by Slaman by combining Theorem~\ref{th:slsoare} with the
Non-Embeddability Condition (NEC) of Ambos-Spies and Lerman~\cite{AL86}. This
last result also suggests that the full answer to our Extendibility
Questions~\ref{qu:extendqu} and~\ref{qu:extendqu01} is likely to be very
hard.

\section{Requirements and Priority Tree}
In this section we describe a set of requirements that guarantee our main
theorem, and the way these requirements can be assigned to strategies on a
priority tree. This methodology is rather standard for priority arguments of
this type, and the reader is referred to the arguments in~\cite{FeSo81, SS93}
(Harrington's plus-cupping theorem and Slaman's triple) which exhibit certain
similarities. Moreover, these ideas are refinements of certain devices that
were used in Lachlan's original $\mathbf{0}'''$\nbd priority argument
in~\cite{La75}. We will also refer to these constructions in
Section~\ref{se:overview}, in order to explain the origins of the basic
strategies for meeting our requirements.

\subsection{List of requirements}\label{subse:listreq}
As usual, we construct two c.e.\ sets~$A$ and~$B$ such that in the end
$\mathbf{a}=\deg(A)$ and $\mathbf{b}=\deg(B)$. We first have the requirements
which satisfy that~$\mathbf{a}$ and~$\mathbf{b}$ form a strong minimal pair:
$$
\cR_i:\ \Phi_i(A)=W_i\ \Rightarrow
\left[\exists \Gamma (\Gamma(B\oplus W_i)=A) \vee
\exists \Delta (\Delta=W_i)\right].
$$

Then we have the diagonalization requirements which guarantee that~$A$ is not
below~$B$ [$B$ not being below~$A$ will be guaranteed automatically]:
$$
\cS_i:\ \Psi_i(B)\neq A.
$$

Note that each~$\cS_i$ states that there exists an~$x$ such that
$\Psi_i(B;x)\neq A(x)$. In the construction, each $\cS_i$\nbd node has
subsidiary $\cS_{i,j}$\nbd nodes, each using a possibly different
\emph{killing point} (to be defined and clarified later) for
forcing~$\Psi_i(B;x)$ to diverge. We call a node associated with such~$\cS_i$
a \emph{parent node} and a node associated with~$\cS_{i,j}$ a \emph{child
node}. At each stage, the collection of an $\cS_i$\nbd parent node and its
previously visited, uncanceled child nodes is called an $\cS_i$\nbd
\emph{family} (of that stage).

\subsection{Discussion of the requirements in a historical context}%
\label{se:hisdi} It is worth noting the similarity of the requirements with
those of the arguments in~\cite{La75, FeSo81, SS93}. Such a discussion may be
beneficial to the reader who is familiar with these older and simpler
arguments; but it may also be helpful to the reader who is not an expert
in~$\mathbf{0}'''$\nbd priority arguments and might like to first consult
these simpler proofs. In its simple form, Harrington's plus-cupping theorem
(presented in~\cite{FeSo81} but also in~\cite{Sho90}) asserts the existence
of a nonzero degree~$\mathbf{a}$ such that every noncomputable
$\mathbf{w}\leq \mathbf{a}$ cups to~$\mathbf{0}'$ (i.e., there exists some
$\mathbf{b}<\mathbf{0}'$ such that $\mathbf{0}'\leq \mathbf{a}\cup
\mathbf{b}$). The main requirements for this theorem (excluding the
noncomputability of~$A$) can be written as
$$
\cR^{\ast}_i:\ \Phi_i(A)=W_i\ \Rightarrow
\left[\exists \Gamma, B_i\ (\Gamma(B_i\oplus W_i)=A\ \wedge
    \ \emptyset'\not\leq_T B_i) \vee
\exists \Delta (\Delta=W_i)\right].
$$
The similarity of the plus-cupping requirements with our requirements of
Section~\ref{subse:listreq} is clear. The main difference is that in the
plus-cupping requirements, for each~$W_i$ we can build a different~$B_i$
while in our requirements there is a unique~$B$ that must accommodate all
conditions. Another relevant example is the construction of a so-called
`Slaman triple', i.e., three degrees $\mathbf{a}, \mathbf{b}, \mathbf{c}$
such that $\mathbf{a}>\mathbf{0}$, $\mathbf{c}\not\leq\mathbf{b}$ and for all
noncomputable $\mathbf{w}\leq\mathbf{a}$ we have $\mathbf{c}\leq
\mathbf{w}\cup\mathbf{b}$. This was published in~\cite{SS93} (based on some
unpublished notes of Slaman from 1983) and it is clear that if we also
require $\mathbf{a}=\mathbf{c}$ then a Slaman triple becomes the strong
minimal pair of our main theorem. The requirements for a Slaman triple
(excluding the noncomputability of~$\mathbf{a}$ and
$\mathbf{c}\not\leq\mathbf{b}$, which is similar to our positive
$\mathcal{S}_i$\nbd requirements in Section~\ref{subse:listreq}) are
$$
\cR^{\ast\ast}_i:\ \Phi_i(A)=W_i\ \Rightarrow
\left[\exists \Gamma\ (\Gamma(B\oplus W_i)=C) \vee
\exists \Delta (\Delta=W_i)\right].
$$
The similarities of $\cR^{\ast\ast}_i$ with our~$\cR_i$ are also clear.
Instead of using the same set for the roles of~$A$ and~$C$ we use two,
therefore relaxing the conflict that is generated between the positive and
the negative requirements. On the other hand, we use a single~$B$ here, in
contrast with~$\cR^{\ast}_i$, where we had a different~$B_i$ for each
condition~$\cR^{\ast}_i$.

The strategies used in the arguments in~\cite{FeSo81, SS93} involve a
gap-cogap technique for the construction of the Turing reductions~$\Gamma$,
which originated in~\cite{La75} and which will also be used in our argument.
In Section~\ref{se:overview}, we will discuss this technique in detail, as
well as the additional difficulties that conditions~$\cR_i$ present, which
are the reason for the more complicated approach we eventually take.

\subsection{Priority tree}\label{subse:tree}
Our priority tree is defined top down, i.e., the top node has the highest
priority. Each node has several possible outcomes, prioritized left to right.

Each $\cR_i$\nbd node~$\alpha$ has two outcomes:~$i$ (infinite) and~$f$
(finite). Along the $i$\nbd outcome, we are defining a
functional~$\Gamma_\alpha$ for computing~$A$ from~$B\oplus W_i$. Such a
node~$\alpha$ is \emph{active} at some~$\beta$ below if there is
no~$g_\alpha$\nbd outcome (see below) between~$\beta$ and~$\alpha$ and there
are no~$\alpha'$ and~$\beta'$ with $\alpha' \subset \alpha \subset \beta'
\subset \beta$ such that~$\alpha'$ and~$\beta'$ form a pair (see definition
below).

Each $\cS_i$\nbd parent node~$\beta$ has three outcomes:~$d$
(diagonalization),~$g$ (gap, defined below), and~$w$ (wait). The $g$\nbd
outcome stands for an apparent computation $\Psi_i(B;x)=0$ against which we
cannot diagonalize (i.e., put~$x$ into~$A$ without risking to lose the
computation $\Psi_i(B;x)=0$). We arrange the priority tree in such a way that
immediately following the $g$\nbd outcome of each $\cS_i$\nbd parent node, we
have its first $\cS_{i,0}$\nbd child node.

Each $\cS_{i,j}$\nbd child node~$\beta$ is below the $g$\nbd outcome of its
$\cS_i$\nbd parent node and has outcomes $g_{\alpha_0},\dots,g_{\alpha_k}, c$
(ordered from left to right). Each~$g_{\alpha}$ (which, following convention,
again stands for ``gap'') corresponds to one active $\cR$\nbd node~$\alpha$
above the $\cS_i$\nbd parent node (\emph{not} the $\cS_{i,j}$\nbd child
node), ordered in such a way that if $\alpha \subset \alpha'$,
then~$g_\alpha$ is to the left of~$g_{\alpha'}$. For the nodes extending the
$g_\alpha$\nbd outcome, we say that~$\alpha$ and this child node~$\beta$ form
a \emph{pair}. In addition, we also define a computable function~$\Delta$
along the $g_\alpha$\nbd outcome for computing the set~$W$ corresponding to
the requirement at~$\alpha$. Extending a $g_\alpha$\nbd outcome, we stop
adding new $\cS_{i,k}$\nbd child nodes (and believe that this requirement has
been satisfied forever).
There is only one $c$\nbd outcome ($c$ stands for ``claim'') to the right of
all the $g_\alpha$\nbd outcomes. Extending such an outcome, we continue to
add new $\cS_{i,k}$\nbd child nodes.
Of course, we arrange the priority tree in a reasonable way such that along
every infinite path, each requirement is represented at most once by a
strategy (or pair of strategies, in the case of the $\cR$\nbd requirements)
which is not enclosed by any other pair.

\section{Overview of the Strategies and Their Conflicts}\label{se:overview}
In this section we discuss the basic strategies that are a starting point for
the more complex strategies that are needed for the satisfaction of the
requirements. We start with the standard gap-cogap strategy for the
satisfaction for simple combinations of prioritized conditions, and slowly
build the ideas needed for the general case. Recall that the tree of
strategies grows from the root downwards, so that a strategy node {\em above}
another is of higher priority with respect to the latter one. The main
conflict occurs between the `positive' requirements (or
strategies)~$\mathcal{S}_j$ (which typically put numbers into~$A$ and try to
preserve a $B$\nbd computation by restraining the enumeration of small
numbers into~$B$) and the `negative' requirements~$\mathcal{R}_i$ which
typically facilitate the enumeration of numbers into~$B$, which are often
needed for the rectification of the functional~$\Gamma_i$ that they build.
The latter rectification is needed due to the enumeration of numbers into~$A$
by some positive strategies. In the preliminary Sections~\ref{subse:SoR}
and~\ref{subse:SmR}, we assume that strategy~$\mathcal{S}_j$ operates from a
single node, instead of being split into a parent node and child nodes as we
described in Section~\ref{subse:tree}. We do this for simplicity, as these
sections only serve as an illustration of the typical gap-cogap strategy,
which is sufficient for simple configurations of requirements but not for the
full construction.

\subsection{Typical gap-cogap strategy: one \texorpdfstring{$\mathcal{S}$}{S}
below one \texorpdfstring{$\mathcal{R}$}{R}}\label{subse:SoR} The strategy of
an $\mathcal{R}_i$\nbd node is to simply enumerate $\Gamma_i$\nbd
computations for the reduction $\Gamma_i(B\oplus W_i)=A$, and enumerate a
number into~$B$ when there is a number~$k$ such that $\Gamma_i(B\oplus W_i;
k)=A(k)$. In the latter case, this number would typically be the current use
of the computation $\Gamma_i(B\oplus W_i; k)$, and its enumeration
facilitates the rectification of the reduction. When~$k$ is the witness of
some positive requirement (or some related parameter, see below) then the use
of the rectified computation may need to be increased to a large number (for
reasons that will become clear when we discuss the $\mathcal{S}_j$\nbd
strategy).

The~$\mathcal{S}_j$ typically picks a witness~$x$ and waits for the
computation $\Psi_j(B; x)$ to converge with output value~0. If and when this
happens, a typical diagonalization strategy would prompt for the enumeration
of~$x$ into~$A$ and the preservation of the $B$\nbd use of the computation
$\Psi_j(B; x)$. However, this naive strategy is not successful in the
present context, since the higher-priority $\mathcal{R}_i$\nbd strategy may
enumerate into~$B$ a number that can destroy the computation $\Psi_j(B; x)$.
Such an enumeration may be caused due to the enumeration of~$x$ into~$A$
by~$\mathcal{S}_j$, and the instructions of~$\mathcal{R}_i$ to maintain the
correctness of~$\Gamma_i$. This is the primary conflict between the
requirements, and at this elementary level it can be resolved by a standard
gap-cogap strategy on the behalf of~$\mathcal{S}_j$ (much like in the
arguments in~\cite{FeSo81, SS93} which we discussed in
Section~\ref{se:hisdi}).

The gap-cogap strategy for~$\mathcal{S}_j$ typically operates in cycles,
periodically restraining~$A$ or~$B$, thus building a potential
computation~$\Delta$ for the set~$W_i$. Prior to the start of the alternating
cycles, it chooses a witness~$x$. The first step in each cycle is:
\begin{itemize}
\item[(w)]
Wait for the computation $\Psi_j(B; x)$ to converge with output value 0.
\end{itemize}
If and when this happens, it checks if the $B$\nbd use of the computation
$\Gamma_i(B\oplus W_i)$ is less than the $B$\nbd use of the computation
$\Psi_j(B; x)$. If this is not true, then it can safely enumerate~$x$
into~$A$ and restrain the $B$\nbd use of the computation $\Psi_j(B; x)$,
thereby securing the disagreement $\Psi_j(B; x)\neq A(x)$. Note that in this
case, the $\Gamma_i$\nbd rectification that may be prompted
by~$\mathcal{R}_i$ will not affect this diagonalization. Otherwise, it will
consider the $W_i$\nbd use of $\Gamma_i(B\oplus W_i; x)$, say~$u_w$, and
\begin{itemize}
\item[(a1)] drop any restraint on~$A$ (thus allowing~$W_i$ to change, under
the assumption that $\Phi_i(A)=W_i$);
\item[(a2)] define $\Delta=W_i$ up to~$u_w$ and restrain enumerations
into~$B$ up to the $B$\nbd use of $\Psi_j(B; x)$.
\end{itemize}
This action initiates an interval of stages that may be called an `$A$\nbd
gap', which is characterized by a lack of restraint on~$A$ and the
enforcement of a restraint on~$B$. During this interval of
stages,~$\mathcal{R}_i$ receives the instruction to increase the use of
$\Gamma_i(B\oplus W_i; x)$ to a large number (larger than the use of
$\Psi_j(B; x)$) in the event that~$W_i$ changes below~$u_w$. When the
strategy is revisited (as in a standard tree of strategies argument),
\begin{itemize}
\item[(b1)]
if~$W_i$ has changed below~$u_w$ since the stage the $A$\nbd gap was
opened, it enumerates~$x$ into~$A$, while enforcing a permanent restraint
on~$B$ for the preservation of $\Psi_j(B; x)\neq A(x)$;
\item[(b2)]
otherwise, it closes the $A$\nbd gap (thereby reinforcing a restraint
on~$A$, equal to the use of the current, possibly new computation
$\Phi_i(A)=W_i$ up to~$u_w$), and opens a $B$\nbd gap by dropping the
restraint on~$B$ and enumerating the $B$\nbd use of $\Gamma_i(B\oplus
W_i; x)$ into~$B$.
\end{itemize}
Note that step (b2) is possible since~$\mathcal{S}_j$ works under the
assumption that the reduction $\Phi_i(A)=W_i$ has infinitely many
expansionary stages. Moreover, note that the $B$\nbd enumeration in step (b2)
will destroy the computation $\Psi_j(B; x)$. Now let us review the long-term
behavior of the~$\mathcal{S}_j$. The routine comes to halt if one of the
following cases occurs at some stage~$s_0$:
\begin{itemize}
\item[(1)] $\Psi_j(B; x)$ remains undefined or not equal to~$0$ at all stages
larger than~$s_0$; or
\item[(2)] $x$ is enumerated into~$A$ by~$\mathcal{S}_j$.
\end{itemize}
In the first case,~$\mathcal{S}_j$ is clearly satisfied (this can be viewed
as a $\Sigma^0_2$\nbd outcome). In the second case, according to the
strategy, the disagreement $\Psi_j(B; x)\neq A(x)$ will be preserved (since
the $B$\nbd use of $\Gamma_i(B\oplus W_i; x)$ would be larger than the
$B$\nbd use of $\Psi_j(B; x)$). Hence in this case also (assuming that basic
priority is respected amongst the requirements)~$\mathcal{S}_j$ is met in a
$\Sigma^0_2$\nbd way. The interesting case is when these events do not occur,
in which case the following cycle of `states' of the $\mathcal{S}_j$\nbd
strategy repeats indefinitely:
\begin{equation}\label{eq:cycle}
\textrm{(w)} \to
\textrm{(a1)} \to \textrm{(a2)} \to
 \textrm{(b2)} \to \textrm{(w)} \to \cdots
\end{equation}
Under this infinitary $\Pi^0_2$\nbd scenario, the witness~$x$ remains fixed,
while the $\mathcal{S}_j$\nbd strategy alternates between $A$\nbd gap states
(when $B$\nbd restraint is imposed but not $A$\nbd restraint) and $B$\nbd gap
states (when $A$\nbd restraint is imposed but not $B$\nbd restraint). The
$A$\nbd gap interval consists of the steps (w), (a1), (a2) (where the latter
two typically occur at the same stage) while the $B$\nbd gap interval
consists of step (b2). In this case, observe that the $\mathcal{S}_j$\nbd
strategy builds a total computable function~$\Delta$ which correctly
computes~$W_i$: new computations are produced at the (a2) steps, and
throughout the stages none of these computations are falsified. Indeed, if
such a computation were falsified (through a $W_i$\nbd change below the
maximum initial segment of numbers on which~$\Delta$ is defined) then the
strategy would execute step (b1), thus ending the perpetual
cycle~\eqref{eq:cycle} and producing a successful $\Sigma^0_2$\nbd outcome
for~$\mathcal{S}_j$. On the other hand, under this outcome, the use of
$\Gamma_i(B\oplus W_i; x)$ is driven to infinity, thereby making~$\Gamma_i$
partial at the chosen number~$x$. This aspect of the strategy is sometimes
known as `capricious destruction' of~$\Gamma_i$, since our strategy
intentionally `kills' the very reduction that we build at a higher-priority
node (but for good reasons, see the next paragraph).

Hence, under this infinitary $\Pi^0_2$\nbd outcome of~$\mathcal{S}_j$ (often
called a `gap outcome'), the actions of this strategy satisfy the
higher-priority~$\mathcal{R}_i$, as well as itself since the use of
$\Psi_j(B; x)$ is driven to infinity. On the other hand,~$\mathcal{S}_j$ can
pass the information that~$\Gamma_i$ is partial at~$x$ to the lower-priority
requirements, so a lower~$\mathcal{S}_{j'}$ can successfully implement a
standard diagonalization strategy by only considering computations
$\Psi_{j'}(B; y)$ which have use $B$\nbd use below the $B$\nbd use of
$\Gamma_i(B; x)$ (which goes monotonically to infinity). In the next section,
we see that this gap-cogap strategy also works in a nested environment, thus
satisfying~$\mathcal{S}_j$ below any finite number of $\mathcal{R}_i$\nbd
strategies.

\subsection{Typical gap-cogap strategy: one \texorpdfstring{$\mathcal{S}$}{S}
below many \texorpdfstring{$\mathcal{R}$}{R}}\label{subse:SmR} When an
$\mathcal{S}_j$\nbd strategy works below a finite number of
$\mathcal{R}_i$\nbd strategies, it needs to resolve the same issues as the
ones discussed in Section~\ref{subse:SoR}, but this time with respect to each
of the higher-priority strategies. More specifically, it may have trouble
preserving a diagonalization $\Psi_j(B; x)\neq A(x)$ due to a number of
$\Gamma$\nbd reductions that are being built with higher priority. In this
section, we show that a nested version of the strategy we discussed in
Section~\ref{subse:SoR} suffices to deal with these conflicts. This nesting
approach is also typical in arguments like those in~\cite{FeSo81, SS93}. For
simplicity, suppose that a node working for~$\mathcal{S}_0$ is below a node
for~$\mathcal{R}_1$, which in turn is below a node working
for~$\mathcal{R}_0$. The methodology we give below generalizes trivially to
the case where we have a node for~$\mathcal{S}_0$ below nodes for
$\mathcal{R}_k,\dots, \mathcal{R}_0$. The idea is to implement the gap-cogap
strategy for~$\mathcal{S}_0$ sequentially, first with respect
to~$\mathcal{R}_1$ and then with respect to~$\mathcal{R}_0$.

Consider the gap-cogap strategy of~$\mathcal{S}_0$ with respect
to~$\mathcal{R}_1$. Under the $\Pi^0_2$\nbd outcome of this strategy,~$W_1$
is proven computable while~$\Gamma_1$ is partial at a specified level (namely
the witness~$x$ of~$\mathcal{S}_0$). In this case, another
requirement~$\mathcal{S}_1$ can work below~$\mathcal{S}_0$, with the
additional information that~$\Gamma_1$ is partial at~$x$. Then a standard
gap-cogap strategy for the copy of~$\mathcal{S}_1$ against~$\mathcal{R}_0$
alone can successfully work for satisfaction of both requirements (as in
Section~\ref{subse:SoR}).

On the other hand, there is a possibility that this gap-cogap routine
of~$\mathcal{S}_0$ against~$\mathcal{R}_1$ ends up having a $\Sigma^0_2$\nbd
outcome. In this case, the strategy would typically go to step (b1). However,
at such a stage,~$\mathcal{S}_0$ can no longer proceed directly with the
diagonalization $\Psi_j(B; x)\neq A(x)$. Indeed, the
higher-priority~$\mathcal{R}_0$ could potentially destroy such a disagreement
(in a way that we have already discussed: through a rectification of
its~$\Gamma_0$ reduction). In this case,~$\mathcal{S}_0$ needs to start a new
gap-cogap cycle with respect to~$\mathcal{R}_0$. If this nested cycle repeats
indefinitely, it provides a computation~$\Delta_0$ for~$W_0$ while making
both~$\Gamma_0$ and~$\Gamma_1$ partial at~$x$. In this case, the highest
priority~$\mathcal{R}_0$ is met, at the expense of~$\mathcal{R}_1$
and~$\mathcal{S}_0$ which are `injured' and need to be satisfied by means of
additional copies of their strategies under the information that~$\Gamma_0$
is partial at~$x$. This is certainly possible, as it reduces to the cases we
have already discussed. If, on the other hand, the second (nested) gap-cogap
cycle of~$\mathcal{S}_0$ reaches step (b1), then it can diagonalize, thereby
producing the disagreement $\Psi_j(B; x)\neq A(x)$ and preserving it
indefinitely (since the relevant $\Gamma_0$- and $\Gamma_1$\nbd uses are
sufficiently large, due to the $W_0$- and $W_1$\nbd changes that occurred,
respectively).

We may sum up the nesting of the gap-cogap strategies as follows.
Strategy~$\mathcal{S}_0$ first attempts to `clear' the computation $\Psi_j(B;
x)$ from the $\Gamma_1$\nbd use on~$x$. If and when it achieves this (through
a $W_1$\nbd change) it proceeds to clear this computation from the
$\Gamma_0$\nbd use on~$x$. If and when this is achieved, it can successfully
diagonalize. In any other case (except the trivial case when~$\Psi_j(B; x)$
remains undefined or not equal to~$0$), it produces a $\Pi^0_2$\nbd outcome
that enables copies of the existing strategies to satisfy their corresponding
requirements at nodes of lower priority.
It is important to note that in the above scenario, after  the computation $\Psi_j(B;x)$
is cleared from the $\Gamma_1$\nbd use on~$x$, the strategy has one chance to clear it from the
$\Gamma_0$\nbd use on~$x$ (namely, in the next cycle when the $A$-restraints drop). If this fails, the strategy starts the module anew, waiting
again for the convergence of $\Psi_j(B;x)$.

Note that here we have two different
$\Pi^0_2$\nbd outcomes corresponding to the following cases:
\begin{itemize}
\item[(1)] we never clear the $\Gamma_1$\nbd use;
\item[(2)] we clear the $\Gamma_1$\nbd use infinitely often but we never clear
the $\Gamma_0$\nbd use.
\end{itemize}
Also note that we only attempt to clear the $\Gamma_0$\nbd use when we have
already cleared the $\Gamma_1$\nbd use. In this sense, we say
that~$\mathcal{S}_0$ first opens a gap for~$\mathcal{R}_1$ and then
for~$\mathcal{R}_0$.

These nested gap-cogap strategies are sufficient for dealing with one
$\mathcal{S}$\nbd strategy below any finite number of $\mathcal{R}$\nbd
strategies. When we consider multiple $\mathcal{S}$\nbd strategies below a
number of $\mathcal{R}$\nbd strategies, new conflicts occur, which we discuss
in the following sections.

Now in our formal construction (see Section~\ref{AB}), we instead handle the
gap-cogap requirements from different notes by alternating global $A$\nbd
stages and $B$\nbd stages in the background. During $A$\nbd stages, we are
allowed to change~$A$ but not~$B$; during $B$\nbd stages, we are allowed to
change~$B$ but not~$A$. Later in the discussion, we will use $A$\nbd stages
and $B$\nbd stages instead of the gap-cogap terminology.

In particular, in the above construction, we do not need to make enumerations immediately but can wait for an appropriate stage to perform the action. For example, after we enumerated a diagonalization witness $x$ into $A$ during an $A$-stage, we cannot simultaneously enumerate the $\Gamma$-use (for the correction of the $\Gamma$ functional computing $A$) into $B$, but we can do this later when we next time visited the corresponding $\mathcal{R}$ node.

\subsection{A minimal new example: two \texorpdfstring{$\mathcal{S}$}{S}
below two \texorpdfstring{$\mathcal{R}$}{R}}\label{sec:min-ex} Here, we
illustrate the idea by a minimal example where we see a conflict which needs
some new strategy, and we will briefly explain how to handle the conflict.
(See Figure~\ref{fig1}.)

\begin{figure}
\centerline{
\xymatrix{\cR_0\ar@{-}[d]^i & & & \cR_0\ar@{-}[d]^i & \\
\cR_1\ar@{-}[d]^i & & & \cR_1\ar@{-}[d]^i & \\
\cS_0(x_0)\ar@{-}[d]^g& & &\cS_0(x_0)\ar@{-}[d]^g&\\
\cS_{0,0}\ar@{-}[d]_{\Delta_1}\ar@{-}[dr]^{c(x_1)} & & &
   \cS_{0,0}\ar@{-}[d]^{c(x')}&\\
\cS_1(x_1)& \cS_{0,1}&& \cR_2\ar@{-}[d]^{i}&\\
 &&&\cS_1(x_1)\ar@{-}[d]^{g}&\\
 &&&\cS_{1,0}\ar@{-}[d]_{\Delta_2}\ar@{-}[dr]^{c(x_2)}&\\
 &&&\cS_{0,1}\ar@{-}[dl]_{\Delta_1}\ar@{-}[d]^{c(x_2)}&\cS_{0,1}\\
 &&\cS_2(x_2)&\cS_3(x_3)&}}
\caption{A minimal example (left) and a complete example (right)}\label{fig1}
\end{figure}
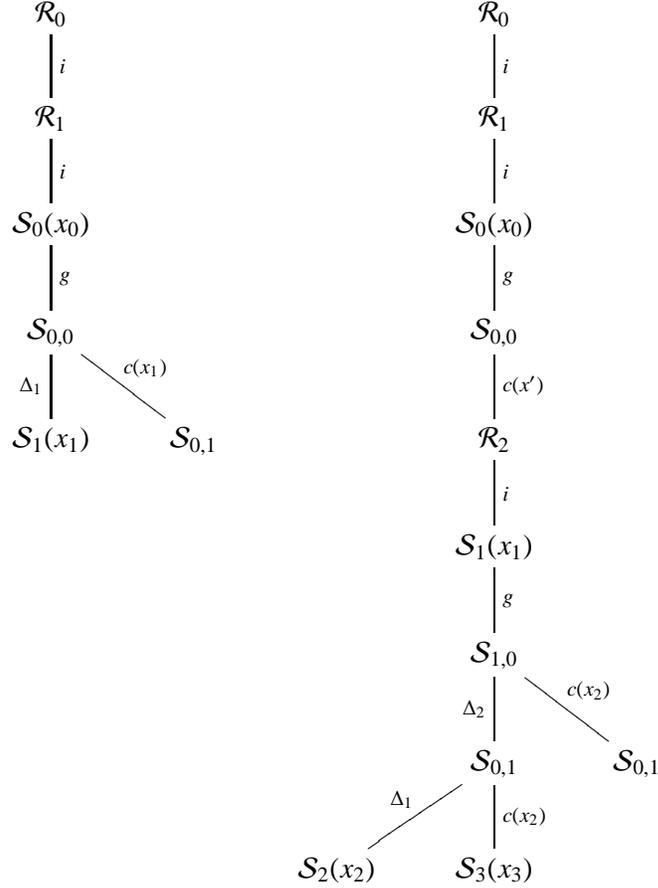

First of all, for later purposes, we want to separate a parent node~$\cS_0$
and its child nodes~$\cS_{0,j}$. Roughly each child node is taking care of
the old strategies which selects the $\cR$\nbd requirement above to pair with
and defines the corresponding function~$\Delta$. The first child
node~$\cS_{0,0}$ is always immediately following its parent node's $g$\nbd
outcome.

Let~$\cR_0$ and~$\cR_1$ be two consecutive $\cR$\nbd requirements, and let
the $\cR_1$\nbd node be extending the $\cR_0$\nbd node's $i$\nbd outcome.
Consider an $\cS_0$\nbd node extending the $\cR_1$\nbd node's $i$\nbd
outcome. Now, at the $\cS_0$\nbd node, as in a usual construction of this
type, we may have a diagonalization witness~$x_0$, but the use $\psi_0(x_0)$
may always be too large (say, $\ge \gamma_1(x_0)$), and so we go to the
$g$\nbd outcome. At the first $\cS_{0,0}$\nbd child node, we use
$\gamma_1(x_0)$ to kill the computation $\Psi_0(B;x_0)$ infinitely often,
say. At the same time, the $\cS_{0,0}$\nbd child node will build a
function~$\Delta_1$ to correctly compute~$W_1$ (for the $\cR_1$\nbd node).

Now, to make sure that~$\Delta_1$ is always correct, the $\cS_{0,0}$\nbd
child node has to set up some mechanism to prevent injury. In the
construction, we implement an alternating $A$\nbd stage/$B$\nbd stage
approach, so that at each stage, at most one~$A$ or~$B$ can change. There are
now two cases here. During a $B$\nbd stage,~$A$ does not change, and so
$W_1=\Phi_1(A)$ (up to the length of agreement) will not change, either,
since otherwise, we will not visit the $\cS_0$\nbd node again. During an
$A$\nbd stage,~$A$ can change but~$B$ does not. If now~$W_1$ changes, then we
can increase the~$\Gamma_1$\nbd use while preserving the $\Psi_0(B;x_0)$\nbd
computation. Then we observe that $\gamma_1(x_0)>\psi_0(x_0)$, and so we will
switch to the left of the outcome associated with~$\Delta_1$. In this
process, unless we move to the left of the outcome associated
with~$\Delta_1$, we see that the $\Psi_0(B;x_0)$\nbd computation is used to
protect~$\Delta_1$ during $A$\nbd stages, since only a $W_1$\nbd change
without a $B$\nbd change guarantees that we can move to the left of the
outcome associated with~$\Delta_1$; so, in the argument, it is crucial that
we can preserve the use of $\Psi_0(B;x_0)$.

Now, say, extending the $\Delta_1$\nbd outcome, we have another $\cS_1$\nbd
node with a witness~$x_1$. During an $A$\nbd stage $s_0$, it might want to
enumerate~$x_1$ into~$A$ for its own diagonalization (and so~$A$ would be
changed). By the observation above, we have to protect the use of
$\Psi_0(B;x_0)$ at the same time. However, if we implement the diagonalization procedure for $\cS_1$ here, then later at $s_1>s_0$ the $\cR_0$\nbd node's
$\Gamma_0$\nbd functional, after observing a change at~$x_1$ in~$A$, will
inevitably add~$\gamma_0(x_1)$ into~$B$ for $\Gamma_0$\nbd correction
(unless~$W_0$ has changed from $s_0$ to $s_1$, but this is not in our control). At $s_1$, however, there is no guarantee that $\gamma_0(x_1)>\psi_0(x_0)$.

The solution is thus briefly as follows: In such a situation at stage $s_0$, we instead go
to a different outcome to the right of the $\Delta_1$\nbd outcome, which we
call the $c$\nbd outcome. We stay at this $c$ outcome as long as $\gamma_0(x_1)\leq\psi_0(x_0)$ (since otherwise, there is no
problem). So at a following stage $s_1>s_0$, as long as $\gamma_0(x_1)\leq\psi_0(x_0)$ is still true, instead of using $\gamma_0(x_0)$ to kill the
$\Psi_0(B;x_0)$\nbd computation, we can use $\gamma_0(x_1)$. We say
that~$x_1$ is the \emph{claim point} for this $c$\nbd outcome at this stage $s_0$.
We count this as a small step toward success. Later at the next $\cS_{0,1}$\nbd
child node, we have a similar scenario for which we may go to the $c$\nbd
outcome with a larger claim point, etc. If this happens infinitely often
along the true path (i.e., there are infinitely many $\cS_{0,j}$\nbd nodes
with a $c$\nbd outcome along the true path), then we are using larger and
larger numbers to push $\psi_0(x_0)$ to infinity, and so the $\cS_0$\nbd
requirement will be satisfied in a $\Pi_3$\nbd way; on the other
hand,~$\Gamma_1$ is still active (since it is only injured finitely often at
each argument), so we do not have to build~$\Delta_1$ for it.

From a local viewpoint, the conflict happens when we see a computation
at~$\cS_1$ which we want to use to diagonalize, and some higher requirements
($\cR_1-\cS_0$) put some restraint on the diagonalization. So the $c$\nbd
outcome with a larger claim point~$x_1$ essentially allows us to freeze the
computation at~$\cS_1$ and at the same time allow~$\cR_1$ and~$\cS_0$ to
continue working towards success by switching the killing point from $x_0$ to $x_1$.

From a global viewpoint, while other outcomes are standard in this type of
gap-cogap construction, each such $c$\nbd outcome is a $\Sigma_2$\nbd type of
outcome, which states that in the construction, there is a stage with a claim
point such that we will keep this claim point (stay in the $c$\nbd outcome)
forever in the following construction.

\begin{table}[ht]
\begin{center}
\arrayrulecolor{green!50!black}
  \begin{tabular}{lcccccc}
{\bf\small Node}& {\bf\small Symbol}&  {\bf\small Access}&
{\bf\small Action} &{\bf\small Sub-action} &{\bf\small Outcomes}
 &{\bf\small Type} \\[1ex]
 \cmidrule{1-7}
{\footnotesize  $\mathcal{R}$} &	{\footnotesize  $\alpha$} &
 {\footnotesize normal} & {\footnotesize defines $\Gamma$}&
{\footnotesize $B$-enumerations} &{\footnotesize $i,f$} &
 {\footnotesize $\Pi^0_2$\ /\ $\Sigma^0_2$} \\[1ex]
{\footnotesize  $\mathcal{S}$-parent}   & {\footnotesize  $\beta$}&
 {\footnotesize normal or child-link}
& {\footnotesize clearing/claim} &
{\footnotesize $A$-enumeration}
 & {\footnotesize $d,g,w$} & {\footnotesize $\Sigma^0_1$\ /\ $\Pi^0_2$\ /\
  $\Sigma^0_2$ }\\[1ex]
 {\footnotesize  $\mathcal{S}$-child}	& {\footnotesize  $\beta_j$}  &
  {\footnotesize n.\ or own-parent-link}
 & {\footnotesize defines $\Delta$}  & {\footnotesize $B$-enumerations} &
 {\footnotesize $g_{\alpha_t}, c$} & {\footnotesize $\Pi^0_2$\ /\
  $\Sigma^0_2$}\\[1ex]
\end{tabular}
\vspace{0.2cm}
\caption{Nodes on the priority tree, their main actions and their outcomes}
\label{ta:vschelparbouatd}
\end{center}
\vspace{-0.8cm}
\end{table}

\subsection{The new idea: \texorpdfstring{$c$}{c}-outcomes}\label{sec:c}

The use of the $c$-outcomes is new and in fact crucial to our construction, so it is important to explain its use and address the differences between a $c$-outcome and a standard $g$-outcome (gap outcome) for example as in Section \ref{subse:SoR}.

As we have mentioned above, the $c$-outcome in our minimal example essentially allows us to freeze the computation at the $\mathcal{S}_1$ node (see Figure \ref{fig1}) as well as $\Delta_1$ to the left of the $c$-outcome, while waiting for a later stage when diagonalization is safe to perform (i.e., the $\Gamma_1$-use is large enough). It is important that here we do not perform any enumeration at this $c$-outcome. A natural attempt, which actually fails to work, would be to perform the same gap-cogap operation with the new witness $x_1$ at the $c$-outcome. The reason is that, it is possible that such a witness $x_0$ or $x_1$ at which the $\Gamma_1$-use is used to push $\Psi_0(x_0)$-use may change and possibly go to infinity. All the lower priority nodes, for a successful construction, need to guess at the outcome correctly. However, with only one (or even infinitely many) $c$-outcome where the gap-cogap strategy is performed, it is not possible for the lower priority nodes to know whether the witness will stop increase or go to infinity.

In general, a $c$-outcome is an outcome of a $\mathcal{S}$-child node, but unlike a $g$-outcome (of the same child node) it does not enumerate any elements (into $B$). Instead such enumeration is delayed to $g$-outcomes of other $\mathcal{S}$-children nodes below this $c$-outcome.

Along the $c$-outcome, the $\mathcal{S}$ requirement at the parent node is satisfied as we push its use $\Psi(x)$ to infinity for a fixed $x$ (the diagonalization witness at the parent node). This is the same as along the $g$-outcomes to the left of the $c$-outcome. What makes a $c$-outcome different is that the delay of $B$-enumerations allows us to satisfy the requirement $\mathcal{R}$ by keeping the corresponding $\Gamma$ total. Note that $g$-outcomes always kill such $\Gamma$ functional and so they need to build $\Delta$'s in order to satisfy $\mathcal{R}$.

In addition, such delay also allows us to work on other (lower priority) requirements between the $\mathcal{S}$-children nodes. That is, if a lower priority node is only below a $c$-outcome of the ``$\mathcal{S}$-family'' but not any of the $g$-outcomes of the children nodes, then it believes that $\Gamma$ is still total and so active. As a result, this also requires some minor adjustments to the different numbers used in our standard gap-cogap construction, which we will describe in the complete example below.

\subsection{A complete example: clearing point, killing point and claim
point}\label{sec:compl-ex}

Now we complete the minimal example above to add in all the features of the
construction. In particular, we will explain various numbers used during the
construction. (See again Figure~\ref{fig1}.) Table~\ref{ta:vschelparbouatd}
can be helpful as a guide on the general structure of the argument and the
complexity of the outcomes of the strategies.

Suppose we have two consecutive $\cR$\nbd requirements~$\cR_0$ and~$\cR_1$,
and the $\cR_1$\nbd node is extending the $\cR_0$\nbd node's $i$\nbd outcome.
Extending the $\cR_1$\nbd node's $i$\nbd outcome, we again have an
$\cS_0$\nbd node followed by its first $\cS_{0,0}$\nbd child node. Now,
extending the $c$\nbd outcome of the $\cS_{0,0}$\nbd node (with a claim
point~$x'$), we have an $\cR_2$\nbd node followed (along its $i$\nbd outcome)
by an $\cS_1$\nbd node.

Let~$x_1$ be the diagonalization witness for~$\cS_1$. When we try to
diagonalize against~$\Psi_1$ at stage $s_0$ when we see a convergent computation $\Psi_1(x_1)$, we first need to make sure that
$\psi_1(x_1)<\gamma_i(x_1)$ for $i=0,1,2$; in addition, notice that
the~$\cS_0$\nbd family currently has a $c$\nbd outcome, which means that
extending any outcome of the $\cS_1$\nbd node (e.g., the $d$\nbd outcome),
there will be more $\cS_{0,j}$\nbd child nodes, and later at any stage $s>s_0$ they will possibly enumerate~$\gamma_0(x')$ or~$\gamma_1(x')$ into~$B$ in order to push
$\psi_0(B;x_0)$ to infinity. This means that, for successful diagonalization
against~$\Psi_1$, we also need to care about possible $\Gamma$-use enumerations at~$x'$ (which is~$<x_1$). So here at stage $s_0$
we call such a number~$x'$ the \emph{clearing point} at~$\cS_1$ and use it to
clear the computation: For clearance, we require $\psi_1(x_1)<\gamma_i(x')$
for $i=0,1,2$. If this is not true, then we use~$x'$ (instead of~$x_1$) to
push $\psi_1(B;x_1)$ at the first $\cS_{1,0}$\nbd child node.

Say, at the $\cS_{1,0}$\nbd node, we choose to go along the $g_{\alpha_2}$\nbd
outcome building~$\Delta_2$ (since $\psi_1(x_1)\geq\gamma_2(x')$). Now
extending this $g_{\alpha_2}$\nbd outcome, say, we first have an
$\cS_{0,1}$\nbd child node. As required by the $\cS_{0,0}$\nbd node, the
$\cS_{0,1}$\nbd node uses~$x'$ to push $\psi_0(B;x_0)$ to infinity. We call
such a number~$x'$ the \emph{killing point} at~$\cS_{0,1}$. Say after stage $s_1>s_0$, such a child node also has a $c$\nbd outcome (whose claim point~$x_2$ comes from some
$\cS_2$\nbd node extending one of its~$g$\nbd outcomes, as in the minimal example). Extending such a $c$\nbd outcome, we have an $\cS_3$\nbd parent node, say with
diagonalization witness~$x_3$.

From the $\cS_3$\nbd node's point of view,~$\cR_2$ has been satisfied (by the
$\cS_{1,0}$\nbd node), and~$\cR_0$ and~$\cR_1$ are still active. The clearing
point at the $\cS_3$\nbd node is~$x_2$, because it believes that the new
$\cS_0$\nbd family members will use~$x_2$ instead of~$x'$ as the killing
point. So the $\cS_3$\nbd node checks whether $\psi_3(x_3)<\gamma_i(x_2)$ for
$i=0,1$.

Now suppose this is true, i.e., we have a cleared computation, say at stage $s_2>s_1$. Then, according to the minimal example above, we next want to make sure
that~$\Delta_2$ is preserved, and we try to clear the $\Psi_1(B;x_1)$\nbd
computation by going to the $c$\nbd outcome of the $\cS_{1,0}$\nbd node.

The tricky part is that, this time at stage $s_1$, for successful clearance, we actually
want $\psi_1(x_1)<\gamma_i(x_2)$ (for $i=0,1$) (instead of
$\psi_1(x_1)<\gamma_i(x_3)$): The reason here is that, to the right of
this~$\Delta_2$, later at any stage $s>s_2$ it is possible that a new $\cS_{0,j}$\nbd child node will use~$x_2$ as the killing point and enumerate~$\gamma_0(x_2)$
or~$\gamma_1(x_2)$ into $B$, and we do not want these numbers to injure
$\Psi_1(B;x_1)$, which we use to protect~$\Delta_2$. We say that~$x_2$ is the
\emph{claim point} of this $c$\nbd outcome at stage $s_2$ (later this claim point is used as the killing point for new~$\cS_{1,k}$\nbd child nodes). When we go to the $c$\nbd
outcome, i.e., the $\Psi_1(B;x_1)$\nbd computation is not cleared, then the
associated \emph{claim} here is that after this stage, it is always the case
that we do not get a clearance, i.e., it is always the case that
$\psi_1(x_1)\geq \gamma_i(x_2)$ for $i=0$ or~$1$.

\begin{table}[ht]
\begin{center}
\arrayrulecolor{green!50!black}
  \begin{tabular}{lccc}
{\bf\small Point}& {\bf\small $\mathcal{S}$-node}&
 {\bf\small Outcome}&{\bf\small Complexity}\\[1ex] \cmidrule{1-4}
{\footnotesize Witness}& {\footnotesize Parent }&
 {\footnotesize All}&{\footnotesize $\Sigma^0_1$}\\[1ex]
{\footnotesize Clearing}& {\footnotesize Parent}&
 {\footnotesize All}&{\footnotesize $\Sigma^0_1$}\\[1ex]
{\footnotesize Claim}& {\footnotesize Child}&
 {\footnotesize $c$}&{\footnotesize $\Pi^0_2$}\\[1ex]
{\footnotesize Killing}& {\footnotesize Child}&
 {\footnotesize $c$}&{\footnotesize $\Pi^0_2$}\\[1ex]
\end{tabular}
\vspace{0.2cm} \caption{Parameters of the $\mathcal{S}$\nbd nodes (parents
and children), associated outcomes and their complexity modulo
initialization.} \label{ta:vschelpauatd}
\end{center}
\vspace{-0.8cm}
\end{table}

\subsection{Overview of the \texorpdfstring{$\mathcal{S}$}{S}-strategies}
\label{sec:dynamicsS}
Table~\ref{ta:vschelpauatd} summarizes the parameters we have introduced for
the $\mathcal{S}$\nbd nodes (parents and children). In this section, we
summarize their dynamics and basic features, in a top-down description (as
opposed to the bottom-up motivational discussion of
Section~\ref{sec:compl-ex}). The diagonalization is done at the parent node,
with a witness which is fixed, as long as the parent node is not injured. The
same is true of the {\em clearing point}, which is another parameter of the
parent node. The clearing point is always less than or equal to the witness.
In the simple case that we described in Section~\ref{sec:min-ex}, we use the
witness as a clearing point, but in the presence of more requirements, we
need to differentiate between the two. The clearing point is the number on
which we may force the associated $\Gamma$\nbd functional to be partial.

Associated with the $c$\nbd outcome of each $\mathcal{S}_{ij}$\nbd child node
is the {\em claim point} of the node. Each time that the $c$\nbd outcome is
activated, it may have a different claim point. Each $\mathcal{S}_{ij}$\nbd
child node also has a {\em killing point}, which is calculated from the claim
points of the higher-priority child nodes. In this way, the killing points of
child nodes are raised according to the claim points of the higher-priority
child nodes with $c$\nbd outcomes. The $c$\nbd outcome of a child
node~$\beta_j$ is initiated by a parent node below~$\beta_j$ (not its own
parent).

\begin{table}[ht]
\begin{center}
\arrayrulecolor{green!50!black}
  \begin{tabular}{lllc}
{\bf\small Satisfaction of $\mathcal{S}$}& {\bf\small Main outcome}&
 {\bf\small Outcome}& {\bf\small Complexity}\\[1ex] \cmidrule{1-4}
{\footnotesize $\Psi(B; x)\un$ co-finitely} &{\footnotesize wait
 outcome (parent)} &{\footnotesize $\Gamma$ total}&{\footnotesize
 $\Sigma^0_2$} \\[1ex]
{\footnotesize $\Psi(B; x)\de\neq A(x)$ co-finitely} &
 {\footnotesize diagonalization (parent)}&{\footnotesize $\Gamma$ total}
 &{\footnotesize $\Sigma^0_2$}\\[1ex]
{\footnotesize $\Psi(B; x)\un$ infinitely often} & {\footnotesize gap outcome
 (child)}&{\footnotesize $\Gamma$ partial} & {\footnotesize $\Sigma^0_3$}
 \\[1ex]
{\footnotesize $\Psi(B; x)\un$ infinitely often}  & {\footnotesize all
 children true $c$-outcomes}&{\footnotesize $\Gamma$ total}&
 {\footnotesize $\Pi^0_3$} \\[1ex]
\end{tabular}
\vspace{0.2cm} \caption{Four different ways that requirement~$\mathcal{S}$
with witness~$x$ may be satisfied, and their complexity relative to the
corresponding parent node.} \label{ta:vpsytd}
\end{center}
\vspace{-0.8cm}
\end{table}

Along with the $c$\nbd outcome, an $\mathcal{S}_{ij}$\nbd child node
implements a gap-cogap strategy, sequentially with respect to the
$\Delta$\nbd functionals of higher-priority child nodes. This gap module
looks for appropriate changes in the approximation to the corresponding
sets~$W$, starting from the closest and moving monotonically toward the root
of the tree. The usual gap-cogap operation of a child node may be interrupted
by its $c$\nbd outcome infinitely often. Infinitely many $c$\nbd outcomes
along the child nodes of a parent node (in the `true path') means that the
functional we try to diagonalize against is partial. Table~\ref{ta:vpsytd}
displays all the different ways that requirement~$\mathcal{S}$ can be
satisfied. The first three ways displayed are typical to a gap-cogap
argument. However, the last case is special and corresponds to the case when
all children fail to succeed with their gap-cogap strategy. In that case,
$\Psi(B;x)$ becomes partial due to the enumeration of $\Gamma$\nbd uses on
larger and larger arguments. Table~\ref{ta:vpsytd} also displays the effect
that the outcomes have on the functional~$\Gamma$ that we build
for~$\mathcal{S}$. Note that in the context of the global construction, where
many requirements are present, the global outcomes are slightly more complex
(e.g.\ a $\Gamma$\nbd functional that is left intact by some child node may
end up partial due to a child of another parent).

\section{Construction}\label{sec:construction}

\subsection{Accessible path, stage dichotomy, accessible nodes and visited
nodes}\label{AB}

In the construction, each stage is either an $A$\nbd stage or a $B$\nbd
stage. We can arrange that all even stages are $A$\nbd stages and all odd
stages are $B$\nbd stages. During $A$\nbd stages, we are allowed to
change~$A$ but not~$B$; during $B$\nbd stages, we are allowed to change~$B$
but not~$A$. Each node first ignores the stage setting and follows the
construction. When the node wants to change~$A$ or~$B$, it checks whether the
current stage setting allows this action. If so, it changes~$A$ or~$B$ as
planned; if not, it terminates the stage and waits.

In addition, each node must try to pass down alternating $A$\nbd stages and
$B$\nbd stages along its (believed) true outcome. If the stage setting is not
the one expected, the node needs to wait for another stage to go to the
outcome we want. For instance, if a node needs to go to an outcome, and at
the last stage that outcome was accessible was an $A$\nbd stage, then we are
expecting a $B$\nbd stage this time. If this is a $B$\nbd stage, then there
is no problem; if this is an $A$\nbd stage, then we terminate the stage.

Now, in these two cases when we terminate the stage (since the stage is not
the one we wanted), at the very next stage (notice that the stage has changed
from~$A$ to~$B$ or from~$B$ to~$A$), we first check whether any~$W$ has
changed (from the previous stage) for those~$W$'s along the accessible path,
up to the previous length of agreement. If so, then for the highest one, we
switch to the $f$\nbd outcome if the length of agreement has decreased (and
it is easy to see that then we have a permanent win unless the node is
initialized), or to the $i$\nbd outcome if the length of agreement increased
(and so we switch to the left if we went to the $f$\nbd outcome at the
previous stage). Otherwise (if there is no $W$\nbd change, or the length of
agreement does not change, or the length of agreement has increased and we
went to the $i$\nbd outcome at the previous stage), then we directly go
through the same accessible path and continue the construction at the node
where we terminated the stage.\footnote{The intuition is that, since no one
has changed~$A$ or~$B$ from the last stage, and the~$W$'s have not changed,
either, unless we can diagonalize, all the uses of computations remain the
same. (See Lemma~\ref{lemAB} for the full proof later on.)} So either we can
change~$A$ or~$B$ as planned, or we can go to the outcome we wanted. In other
words, at each node, if the last stage was a terminated stage and there is no
$W$\nbd change, then we continue to the same outcome without any extra
action.

As in a usual priority tree construction, at each stage~$s$, we inductively
construct an \emph{accessible} path (up to length~$s$) on the priority tree.
At each node along the accessible path, we try to decide the outcome at
stage~$s$ and whether we want to change~$A$ or~$B$. Whenever~$A$ or~$B$ is
changed, we terminate the current stage and go to the next stage. We keep the
nodes that are to the left of, or compatible with, the accessible path and
initialize the nodes that are to the right.
Note that we may build a link in the construction and skip some nodes along
the accessible path (without going through the construction for them at that
stage). So we shall distinguish between notions of a node being visited and
being accessible. Being \emph{visited} means that we allow this
node to act according to the construction below; and being \emph{accessible}
only means that the node is on the accessible path, which does not
necessarily mean that the node itself is visited but possibly only some
extension of it is.

In the following subsections, we always assume that we are at a visited node
at stage~$s$.

\subsection{\texorpdfstring{$\cR$-node}{R-node}}

Consider an $\cR$\nbd node~$\alpha$ and note that if the last stage was a
terminated stage and~$W$ has not changed, then we continue to the same
outcome without any action. Otherwise, we check whether the \emph{length of
agreement} has increased since the last stage~$t$ when we visited this node
and the $i$\nbd outcome was accessible (or if such a stage~$t$ does not
exist, then we check whether the length of agreement is positive). If not,
then we go to the $f$\nbd outcome. If so, then we go to the $i$\nbd outcome.

The $\cR$\nbd node~$\alpha$ also defines a
functional~$\Gamma$ along the $i$\nbd outcome. We make sure that~$\Gamma$ is
\emph{well-defined}, i.e., we will not enumerate axioms that use the same
oracle but give different outputs. In particular, we may have some
\emph{requests} to add some numbers into~$B$ here which were assigned by
nodes below. What we do is simply put these numbers into~$B$ as planned if
the corresponding~$W$ has not yet changed (see Section~\ref{diag}).

For convenience, we allow the $W$\nbd use and $B$\nbd use for the same~$x$ to
be different (so we formally write $\gamma(W;x)$ and $\gamma(B;x)$ to denote
these uses, but later, when it is clear from the context that we are talking
about the $B$\nbd use, we will simply write~$\gamma(x)$). Since all the sets
we consider are c.e., at each stage we only need to keep one axiom
$\Gamma(B\oplus W;x)$ for a fixed~$x$. We have two cases in which we increase
the use. The first case is that some node below puts $\gamma(B;x)$ into~$B$
but $A(x)=0$; in this case, we increase the $B$\nbd use to be large and
fresh, and increase the $W$\nbd use to be the length of agreement
between~$\Phi(A)$ and~$W$ at this stage. The second case is when the $W$\nbd
use changes; then we increase the $B$\nbd use to be large and fresh and keep
the $W$\nbd use the same. In all other cases, we do not increase the uses but
simply update the axiom with the current oracle.

Of course, we obey the usual monotonicity rules of axioms, that is, whenever
we change the uses for some~$x$, we automatically make $\Gamma(B \oplus W;y)$
undefined for all $y>x$. In any case, we will ensure that $\Gamma(B \oplus
W;x)=A(x)$ for all $x \le$ the current length of agreement between~$\Phi(A)$
and~$W$ at this stage; if a use for $\Gamma(B \oplus W;x)$ had never been
picked before, then we pick the $B$\nbd use large and fresh, and the $W$\nbd
use to be the current length of agreement between~$W$ and~$\Phi(A)$;
otherwise, the use is specified as above.

\subsection{\texorpdfstring{$\cS$-parent node}{S-parent node}}

At an $\cS_i$\nbd node~$\beta$, if this is the first time at which we visit
this node, then we pick a fresh diagonalization witness~$x$ for it. Now if we
already have a diagonalization witness~$x$, then we check whether
$\Psi_i(B;x)$ converges to~$0$ with a \emph{believable} computation. Here, and in the following, a
computation $\Psi_i(B;x)[s]\de$ is believable  when there are no numbers below the use of this computation
that may enter $B$ at a later stage, by the nodes above $\beta$ (such are uses of $\Gamma$-functionals above $\beta$
that are partial from the point of view of $\beta$). If not, then
we go to the $w$\nbd outcome and continue to the next node. If we find out
that earlier we have already visited the $d$\nbd outcome (i.e., we have
already performed diagonalization at this node and $A(x)=1$). and~$\beta$ has
not been initialized since, then we continue to go to the $d$\nbd outcome.

If there is such a believable computation $\Psi_i(B;x)\downarrow=0$ (where,
when we see a believable such computation, we immediately initialize every
node extending the $w$\nbd outcome) but we have not yet performed
diagonalization (i.e., enumerated~$x$ into~$A$), then we perform the
following construction. We first check whether we can perform diagonalization
(see below in Section~\ref{diag}) and if so, follow the instructions; if not,
then we go to the $g$\nbd outcome (or some other outcomes according to
Section~\ref{nondiag} below) and continue to the next node.

\subsubsection{Diagonalization, setting  clearing and claim points}
\label{diag}


At~$\beta$, we consider those $\cS_{i'}$\nbd requirements which have $g$\nbd
outcome along~$\beta$ and none of whose child nodes has a $g$\nbd outcome
along~$\beta$. We think of the $\cS_{i'}$\nbd family as a whole as announcing
the current \emph{killing point for the requirement~$\cS_{i'}$}, which is
defined as the greatest number among all claim points of all $\cS_{i'}$\nbd
child nodes above or to the left of~$\beta$ as well as the clearing point
at~$\cS_{i'}$. Then we let the \emph{clearing point~$y$ at~$\beta$} be the
least of these killing points announced by the $\cS_{i'}$\nbd families from
above as well as~$x$ (if there is no such
higher-priority~$\cS_{i'}$).\footnote{Since~$x$ is a fresh number when it is
picked, this~$y$ is always less than or equal to~$x$ (Lemma~\ref{witness
lemma 1}). Roughly speaking, this~$y$ is going to be the least killing point
when we go to the right of~$\beta$, and so for successful diagonalization, we
want to make sure that~$\beta$'s computation is protected when we switch to
the right of it. In the complete example in Section~\ref{sec:compl-ex},
our~$x$ here is~$x_3$ there, and our~$y$ here is~$x_2$ there.}

We check whether $\gamma_k(y)>\psi_i(x)$ (for the
clearing point~$y$ defined above) for each active~$\cR_k$ above. If not, then
we go down to the $g$\nbd outcome here (see Section~\ref{nondiag}) and, at
the first $\cS_i$\nbd child node~$\cS_{i,0}$, we will go to the corresponding
$g_{\alpha_k}$\nbd outcome defining a function~$\Delta$ and add $\gamma_k(y)$
into~$B$ there (for the greatest such~$k$, see details below in
Section~\ref{child node}). If $\gamma_k(y)>\psi_i(x)$, then we proceed to the
following check.\footnote{If so, note that $y\leq x$, so it is automatic that
$\gamma_k(x)\geq \gamma_k(y)>\psi_i(x)$ and it seems that we are safe to
put~$x$ into~$A$.}

Here, it is possible that for some other~$\Delta'$ defined at an
$\cS_{i'}$\nbd child node~$\beta'$ above~$\beta$ (along the same path), we
use the corresponding $\Psi_{i'}(B;x')$\nbd computation to protect~$\Delta'$,
yet some $\gamma_k(x)$ entering~$B$ for~$\Gamma_k$ above this $\cS_{i'}$\nbd
node may cause injury, i.e., $\gamma_k(x)\leq \psi_{i'}(x')$.

If there is no such~$\beta'$, i.e., for every~$\beta'$ along~$\beta$, we have
$\gamma_k(x)> \psi_{i'}(x')$ as above, then we can put~$x$ into~$A$ and go to
the $d$\nbd outcome of~$\beta$. While doing that, we issue requests at each
active~$\cR$\nbd node above~$\beta$ to add~$\gamma(x)$ into~$B$ as follows:
Later when we visit $\cR$'s $i$\nbd outcome, if the corresponding $W$\nbd use
(for $\Gamma(B \oplus W;x)$) has changed, then we do not add~$\gamma(x)$
into~$B$, but otherwise, we add~$\gamma(x)$ into~$B$.


If we see such~$\beta'$, then fix the lowest (i.e., we process these nodes
from the bottom up) such~$\beta'$ for which $\gamma_k(x)\leq \psi_{i'}(x')$,
we consider all $\cS_{i''}$\nbd nodes above~$\beta'$ which have a $g$\nbd
outcome along~$\beta'$ but such that no child node has a $g$\nbd type outcome
along~$\beta'$ (i.e., the $\cS_{i''}$\nbd requirements that are still active
at~$\beta'$). For each such $\cS_{i''}$\nbd node, we only look at its child
nodes below~$\beta'$ (the $\cS_{i''}$\nbd \emph{family below}~$\beta'$).
These child nodes define a current killing point, i.e., the maximum claim
point (if such $\cS_{i''}$\nbd family below~$\beta'$ is empty, then let this
current killing point be infinity). Then we let the \emph{claim point}~$z$
of~$\beta$ be the minimum number among all these killing points of
$\cS_{i''}$\nbd families below~$\beta'$ (for all such~$\beta'$), as well
as~$x$, the diagonalization witness at~$\beta$. So automatically~$z$ is less
than or equal to~$x$. \footnote{Later we will see that it is automatically
greater than the killing point at~$\beta'$ (Lemma~\ref{witness lemma 2}). In
the complete example in Section~\ref{sec:compl-ex}, our~$z$ here happens to
be~$x_2$ there as well, just like our~$y$ here is~$x_2$ there, but this need
not be true in general.}

This $c$\nbd outcome at~$\beta'$ is now associated with the \emph{claim} that
``after this stage~$s$, it is always the case that~$\psi'(x')$ is greater
than or equal to~$\gamma_k(z)$ for some active~$\Gamma_k$ above the
$\cS_{i'}$\nbd parent node''. (For convenience we denote this claim by
$\mathcal{C}(\beta',z,s)$.) In addition, this $c$\nbd outcome announces
that~$z$ is the new killing point for lower-priority $\cS_{i'}$\nbd child
nodes, overwriting the old announcements made by higher-priority child nodes
for the same~$\cS_{i'}$. That is,~$\cS_{i'}$, as a whole requirement, now
switches the killing point to~$z$. In this case, we say that~$\beta$
\emph{initiates} the $c$\nbd outcome at~$\beta'$.\footnote{Later, when we
reach the parent node for~$\beta'$, we can check whether the condition
$\gamma_k(z)\leq \psi'(x')$ is still true, i.e., whether this claim is still
true; if not, then we will initialize everything extending the $c$\nbd
outcome at~$\beta'$ and declare that this node~$\beta'$ now gives permission
for diagonalization at~$\beta$.} We go to the $c$\nbd outcome of~$\beta'$ and
continue to the next node along that path.


\subsubsection{Possible link to child} \label{nondiag}

Now, at this time, if we do not have a chance to diagonalize, there might be
some $\cS_{i'}$\nbd child nodes below, whose $c$\nbd outcome has been
initiated with a claim about the size of $\psi_{i'}(x')$ and some
$\Gamma$\nbd uses of possibly larger~$x''$ (see above). We check if any of
these claims turn out to be false. For those corresponding $c$\nbd outcomes
whose claims turn out to be false, we initialize everything below the $c$\nbd
outcome of these child nodes and everything to the right of them.

In addition, we check whether there is an $\cS_{i,j}$\nbd child node such
that the last time it was visited we went to one of its $g$\nbd outcomes, and
now with the current conditions we see that we can switch to the left to that
$g$\nbd outcome. If there is such a child node, then we build a link directly
from the $\cS_i$\nbd parent node~$\beta$ to that child node, skipping every
node between them. Otherwise, we stay at the $\cS_i$\nbd parent node~$\beta$
and proceed to the next node along the $g$\nbd outcome.

\subsection{\texorpdfstring{$\cS$-child node}{S-child node}}
\label{child node}

When we reach an $\cS_{i,j}$\nbd child node~$\beta_j$ of an $\cS_i$\nbd
node~$\beta$, the construction proceeds as follows. First, as we have
mentioned above,~$\beta_j$ checks whether the $c$\nbd outcome was accessible
at the last stage~$t$ when we visited~$\beta_j$. If so, we check if the
associated claim $\mathcal{C}(\beta_j,z,t)$ is still true. In that case, we
go down to that outcome without doing anything here. If the claim is false,
then we have already initialized everything extending the $c$\nbd outcome
of~$\beta_j$ when we reach~$\beta$. In that case, there must be some
$\cS$\nbd parent note~$\beta''$ below some $g_\alpha$\nbd outcome
of~$\beta_j$ which initiated the $c$\nbd outcome of~$\beta_j$ here. If this
node~$\beta''$ has not been initialized since, then we directly link to
this~$\beta''$, allowing it to finish trying its diagonalization (without
visiting the nodes between~$\beta_j$ and~$\beta''$). If this~$\beta''$ has
already been initialized, then we proceed as in the following paragraph.

Otherwise, i.e., if we didn't visit the $c$\nbd outcome the last time we
visited~$\beta_j$, then we have a \emph{killing point}~$y$ here decided by
higher-priority $\cS_{i,j'}$\nbd child nodes~$\beta_{j'}$ above or to the
left of~$\beta_j$ (or by~$\beta$ itself if there is no such~$\beta_{j'}$):
$y$~is the largest of all the claim points of these~$\beta_{j'}$ as well as
the clearing point at~$\beta$. We also know that $\gamma(y)\leq \psi_i(x)$
for some functional~$\Gamma$ by some active $\cR$\nbd node above~$\beta$;
let~$\alpha$ be the lowest-priority such $\cR$\nbd node. Now we go to the
$g_\alpha$\nbd outcome. If this is a $B$\nbd stage, we also add $\gamma(y)$
into~$B$. For the functional~$\Delta$ associated with the $g_\alpha$\nbd
outcome, we extend~$\Delta$ up to the $W$\nbd use $\gamma(W;y)$. Then we
continue to the next node, this finishes the inductive step of the accessible
path construction.

\section{Verification}
We start with a few technical lemmas, then
we can show that there is a leftmost path accessible infinitely often (the
\emph{true path}) and every node on the true path has a true outcome. We then
show that all the functionals~$\Gamma$ (unless killed) and all
functions~$\Delta$ built along the true path are well-defined. This allows us
to show that all requirements are satisfied.

\subsection{Technical lemmas}

First of all, in our construction, we separated the stages into $A$\nbd
stages and $B$\nbd stages, and only allowed changes in~$A$ or~$B$ at $A$\nbd
stages or $B$\nbd stages, respectively. Sometimes, we may encounter the
situation that the algorithm wants to change~$A$ but the current stage is a
$B$\nbd stage, or vice versa, and so in the construction, we simply terminate
the stage and immediately try the next stage. (See Section~\ref{AB} for
details.) We start with a lemma proving that in this case, either we will
change the accessible path due to a $W$\nbd change (which will cause either
initialization of the node that wanted to enumerate, or the permanent
satisfaction of the requirement of a higher-priority node), or we can perform
the desired $B$- or $A$\nbd enumeration at the next stage.

\begin{lem}[Accessibility of $A$/$B$-stages]\label{lemAB}
Suppose at stage~$s$, we terminated the stage because the stage was not of
the type we wanted. Then at the next stage $s+1$, either some~$W$ changes and
we switch to the left or right of the accessible path at stage~$s$, or we can
perform the enumeration we wanted to perform at stage~$s$.
\end{lem}

\begin{proof}
According to the construction, assume that some~$W$ along the accessible path
(of stage~$s$) changes at stage $s+1$ by~$x$ entering~$W$: If this change
decreases the length of agreement between~$W$ and~$\Phi(A)$ and switches the
outcome of a strategy along the accessible path at stage~$s$ from an $i$\nbd
outcome to an $f$\nbd outcome, then we have permanent satisfaction of an
$\cR$\nbd requirement (unless some higher-priority node acts), since $W(x)=1$
and we have a computation $\Phi(A;x)=0$. If this change increases the length
of agreement or does not change it, then actually it will not affect any of
the~$\Delta$'s previously defined below the $i$\nbd outcome (since we only
define~$\Delta$ up to the length of agreement). Now, if we do not switch the
accessible path between stages~$s$ and $s+1$, then obviously, since we have
not changed~$A$ or~$B$ from stage~$s$ to stage $s+1$, all criteria required
for action remain the same, and we can perform the action (go to a certain
outcome or change~$A$ or~$B$) as at the previous stage~$s$.
\end{proof}

Usually, in a priority tree argument, one can simply see by inspection
that, for any computation (e.g., of~$\Psi$, $\Phi$) witnessed at a node, the
use cannot be changed by any node to the right of it (by the choice of
sufficiently large witnesses). However, in our construction, this is not
true. The problem is that, along a $c$\nbd outcome of an $\cS$\nbd child
node, the killing point~$z$ is determined by some node extending a
$g_\alpha$\nbd outcome of the $\cS$\nbd child node, i.e., to the left of its
$c$\nbd outcome. Therefore, potentially any $B$\nbd change up to~$\gamma(z)$
via at a node extending the $c$\nbd outcome might injure some $\Psi$\nbd
computations to the left of it. So we need a lemma stating that, in certain
cases, such injury cannot happen.

\begin{lem}[Link to a parent node]\label{injury lemma}
In the construction, if we see that a claim for a $c$\nbd outcome at some
$\cS_{i,j}$\nbd node~$\beta$ becomes false and build a link to an
$\cS_{i'}$\nbd node~$\beta'$ along a $g_\alpha$\nbd outcome (which initiated
the $c$\nbd outcome), then at that time, the computation at~$\beta'$ is still
the same as when~$\beta'$ initiated the $c$\nbd outcome.
\end{lem}

\begin{proof}
Say, at stage~$s_0$,~$\beta'$ initiated the $c$\nbd outcome and by the
criterion in the construction, we know that the use~$\psi(x)$ at~$\beta'$
(for the diagonalization witness~$x$ at~$\beta'$) is $\le \gamma(y)$ for the
least possible killing point~$y$ that can be used to the right of~$\beta'$.
If such~$y$ in the definition decreases (i.e., some node to the right uses a
smaller number as the killing point), then we would have initialized~$\beta'$
and would not build a link from~$\beta$. This means that when we build a link
back to~$\beta'$, its computation is preserved.
\end{proof}

\begin{lem}[Diagonalization of parent preserved]\label{diag pres lemma}
If an $\cS$\nbd node has performed diagonalization, then unless it is
initialized, its computation $\Psi(B;x)$ is always preserved.
\end{lem}

\begin{proof}
The argument is almost the same as the previous lemma. If a killing point~$y$
had decreased, then it would mean that the node had been initialized. If the
killing point has not decreased, then by our criterion, the computation is
preserved.
\end{proof}

\subsection{True path lemmas}
Since our tree is finitely branching, there clearly is a leftmost path
accessible infinitely often (which we call the \emph{true path}). The
slightly tricky problem is that in the construction, there are two cases when
we build a link between two nodes and skip nodes in between: The \emph{first}
case is when an $\cS_i$\nbd node sees that an $\cS_{i,j}$\nbd child node can
now switch to the left; the \emph{second} is from a $c$\nbd outcome of an
$\cS_{i,j}$\nbd node to an $\cS_{i'}$\nbd node below one of its
$g_\alpha$\nbd outcomes. It is conceivable that some node on the true path is
skipped infinitely often but not visited infinitely often, or its outcome is
along the true path but is actually not the \emph{true outcome} (the leftmost
outcome we choose infinitely often when visiting the node). The following few
lemmas show that this case cannot happen. The idea to prove this is as
follows: Each time we skip over a node~$\beta$, we always ``blame'' a node
below it and make sure that such a node can only do this finitely often
before~$\beta$ is visited again.

\begin{lem}[First case skip]\label{true_path_lemma_1}
If a node~$\beta$ is skipped via the first case, then some node below it
switches left. In addition, if~$\beta$ is never visited again and never
skipped by the second case, then the skip for the first case can only happen
finitely often, and each time we will go strictly to the left of the previous
visit.
\end{lem}

\begin{proof}
The first claim follows by inspection of the construction. For the second
claim, note that for every such link which skips~$\beta$,~$\beta$ must be
between an $\cS$\nbd node and one of its child nodes. A somewhat tricky
situation may arise that during such a stage when~$\beta$ is skipped, we may
add new nodes below it which may cause extra links. But observe that such a
new link must be associated with an $\cS'$\nbd parent node of higher priority
than the $\cS$\nbd node which causes the skip at the current stage, so by
induction on the number of $\cS$\nbd parent nodes above~$\beta$, one can see
that, if~$\beta$ is never visited again, such a skip (for the first case) can
only happen finitely often. More precisely, we associate each skip to a
combination of $\cS$- and $\cR$\nbd nodes of higher priority than~$\beta$,
and assign a natural priority on these combinations. It is then easy to check
that each time we go to the left, such a combination increases in priority,
and so this cannot happen forever.
\end{proof}

\begin{lem}[Second case skip]\label{true_path_lemma_2}
At any stage, for any given~$\beta$, there can be at most one node~$\beta'$
below~$\beta$ which has initiated a $c$\nbd outcome at a node above~$\beta$
such that the associated claim is still true. That is, during any fixed
stage, there can be at most one node which makes us skip~$\beta$ for the
second case.
\end{lem}

\begin{proof}
Suppose~$s_0$ is the first stage such that the $c$\nbd outcome of~$\beta'$ is
initiated. Then, of course, at stage~$s_0$, there is only one such node (we
jump to the $c$\nbd outcome at~$s_0$). After that, either~$\beta'$ is
initialized; or the associated claim never becomes false, and so the claim of
the lemma remains true; or later the claim becomes false at stage~$s_1$ and
we build a link directly to~$\beta'$ skipping~$\beta$. At that stage, we note
that the computation at~$\beta'$ is still the same as that at stage~$s_0$ (by
Lemma~\ref{injury lemma}). So at stage~$s_1$, either~$\beta'$ again initiates
another $c$\nbd outcome even higher, or it follows diagonalization and now
there are no nodes which make us skip~$\beta$ (for the second case). The same
situation happens at every stage afterwards, and so the lemma follows.
\end{proof}

\begin{lem}[True path]
Along the true path, every node is visited infinitely often, therefore
all outcomes along the true path are true outcomes.
\end{lem}

\begin{proof}
This follows essentially by combining Lemmas~\ref{true_path_lemma_1}
and~\ref{true_path_lemma_2}. Suppose some~$\beta$ on the true path is never
visited again. Whenever we skip~$\beta$ via the second case, then some node
below performs diagonalization, which means that any nodes extending the
$d$\nbd outcome will be fresh at that stage. At that moment, the only reason
we can skip~$\beta$ is the first case, and so the next time we skip
over~$\beta$, we must travel to the left of the current visit. It then
follows that below any of these diagonalization outcomes~$d$, we will not
have new nodes added which request diagonalization, since each such new
$\cS$\nbd node is visited only once.

Therefore we eventually switch to the left of this diagonalization outcome,
and by the same argument as in Lemma~\ref{true_path_lemma_1} above, such
skips cannot happen infinitely often. So one can only skip over~$\beta$
finitely often, and the lemma follows.
\end{proof}

In addition, we need to show that every node along the true path ``passes
down'' infinitely many $A$\nbd stages and $B$\nbd stages (in fact, in
alternating order), so every node has the chance to perform the action it
wants to eventually.

\begin{lem}[Alternating stages on true path]
In the construction, every node on the true path is visited infinitely often
at $A$\nbd stages and at $B$\nbd stages, respectively.
\end{lem}

\begin{proof}
This is because in the construction, we require that when we pass to an
outcome, we require a different type of stage ($A$\nbd stage or $B$\nbd
stage) than the one when we last time went to that outcome (otherwise, we
wait and do nothing). Along the true path, as we proved above, every node is
actually visited infinitely often, and so by this criterion, every node is
visited at alternating $A$\nbd stages and $B$\nbd stages.
\end{proof}

Now in the following arguments, we always assume that we have a node~$\xi$ on
the true path and we have passed the stage when all nodes to the left stop
acting. Here, action include being visited or accessible, or $c$\nbd outcome
initiation. Since there is finite injury along the true
path, we also assume that~$\xi$ is the last node along the true path for its
requirement, and we only consider stages when it is visited.

\subsection{Witnesses and functionals}
First, we prove two lemmas about the witnesses and various other points we
use in the construction.

\begin{lem}[Clearing point and witness of parent node]\label{witness lemma 1}
Given an $\cS$\nbd node with diagonalization witness~$x$, the clearing
point~$y$ (as in the construction) is always less than or equal to~$x$, and
such~$y$ is stable if no node to the left acts again.
\end{lem}

\begin{proof}
This is by inspection of our construction.
\end{proof}

\begin{lem}[Claim and killing point of child node under $c$\nbd outcome]
\label{witness lemma 2}
Given an $\cS_{i,j}$\nbd child node, when its $c$\nbd outcome is initiated
(by~$\beta$, say), the corresponding claim point~$z$ (as in the construction)
is always strictly larger than its killing point, and is always less than or
equal to the diagonalization witness at~$\beta$.
\end{lem}

\begin{proof}
The second claim is by inspection of the definition of such~$z$. The first
claim follows from the fact (proved by induction) that such~$z$ is always a
diagonalization witness below an $\cS_{i,j}$\nbd child node's $g_\alpha$\nbd
outcome (for some~$\alpha$), and so larger than the killing point (whenever
it changes, every node below is initialized automatically).
\end{proof}

Next, we show that along the true path, every functional is correct on its
domain (modulo finite incorrectness for the~$\Delta$'s). It follows that the
functional computes the set we want if it has total domain.

\begin{lem}[$\Gamma$-functionals]\label{lemma gamma}
Every functional~$\Gamma$ is correct on its domain.
\end{lem}

\begin{proof}
This is basically by inspection of the construction that when we add any
number~$x$ into~$A$, we always make sure to issue requests to add the
corresponding $\gamma(x)$\nbd uses into~$B$ at~$\Gamma$. It may be the case
that later when we visit~$\Gamma$, the corresponding~$W$ has changed up to
the use, and since~$W$ is c.e., such a change automatically makes the
functional undefined and so there is no problem in not adding~$\gamma(x)$
into~$B$ in this case. If~$W$ has not changed, then, of course, by the
construction, we will add~$\gamma(x)$ into~$B$ so that we can correct the
axiom.
\end{proof}

The next lemma is going to be the most crucial and most complicated lemma in
the proof. Let us first sketch the argument: To show that $\Delta=W$, it
suffices to show that whenever we define some~$\Delta$ as an initial segment
of~$W$, then this initial segment of~$W$ is not going to change in the
construction later. Now at $B$\nbd stages, this is obvious since $W=\Phi(A)$
where~$A$ does not change. At $A$\nbd stages, the argument is much trickier,
but is very similar to the standard argument used in the style of Lachlan's
gap-cogap construction. Basically, we have a computation $\Psi(B;x)$ to
protect an initial segment of~$W$ in such a way that if it changed (after we
changed~$A$) then we would switch to the left of the $\Delta$\nbd outcome.
The difficult part is to show that after~$A$ changes, the $B$\nbd use
of~$\Psi(x)$ is always protected. This is usually true since we have only
been to the right of such~$\Delta$, but remember that in our construction,
actions to the right may injure computations to the left.

\begin{lem}[$\Delta$-functionals]\label{lemma delta}
Every function~$\Delta$ is correct on its domain (modulo a fixed finite
amount of injury). More precisely, for every such~$\Delta$, there is a stage
after which~$\Delta$ is not going to be injured again.
\end{lem}

\begin{proof}
Say, such~$\Delta$ is defined along a $g_\alpha$\nbd outcome (with killing
point~$x'$) of~$\beta_i$, which is a child node for~$\beta$ (where~$\beta$
has diagonalization witness~$x$).

In addition, we know that, for each parent node~$\beta'$ above~$\beta$ and
active at~$\beta$, every child node of this~$\beta'$ along the true path has
true $c$\nbd outcome. Now we have to wait for a stage~$s_0$ such that every
such~$\beta'$ has a child node below~$\beta$ (on the true path) with a true
$c$\nbd outcome initiated (i.e., a $c$\nbd outcome that will not be
initialized later).

We claim that after stage~$s_0$, the $\Delta$\nbd axioms are always correct,
i.e., compute $W=\Phi(A)$. If~$A$ does not change, then, of course,~$W$
cannot change. So we only need to consider the case when~$A$ changes, in
particular, below~$\beta_i$'s $g_{\alpha'}$\nbd outcomes, since otherwise,
such an $A$\nbd change must be to the right and cannot change the initial
segment of~$W$ witnessed at~$\beta'$.

Suppose that at some later stage~$s_1$, some node~$\bar{\beta}$
below~$\beta_i$'s $g_\alpha$\nbd outcome performs diagonalization (most
likely via a link under the second case). According to the construction, such
a node~$\bar{\beta}$ must receive permission from every child node with
$g_{\alpha'}$\nbd outcome above it. In particular,~$\beta_i$ needs to give
permission that $\gamma(z)>\psi(x)$, where~$z$ is the associated claim point
at~$\beta_i$, and the $\Gamma$\nbd uses range over all~$\Gamma$'s active
above~$\beta$.

By the definition of stage~$s_0$, such~$\gamma(z)$'s are going to be the
least possible numbers entering~$B$ when we switch to the left of~$\beta$;
and by Lemma~\ref{witness lemma 2},~$z$ is less than or equal to the
diagonalization witness added into~$A$. In addition, by inspection of the
construction, we know that at stage~$s_1$, the computation~$\Psi(B;x)$
converges. (Otherwise, the permission criterion $\gamma(z)>\psi(x)$ is always
false.)

So we know that, after we add the diagonalization witness into~$A$ at
stage~$s_1$, and before we come back to~$\beta$, the computation~$\Psi(B;x)$
at~$\beta$ is always preserved. Now it suffices to show that $W=\Phi(A)$ up
to~$\gamma(x')$ is preserved (recall that~$x'$ is the killing point
at~$\Delta$ and we always define~$W$ up to~$\gamma(x')$).

Otherwise, when we reach the $\cR$\nbd node and go to its $i$\nbd outcome, we
would see that the use $\gamma(W;x')$ has changed, and so according to the
construction, we will increase its $B$\nbd use without changing~$B$ here. In
particular, we know that when we reach~$\beta$ for the first time
after~$s_1$, $\gamma(x')>\psi(x)$, and according to the construction at
$\cS$\nbd nodes, we would immediately build a link to this~$\beta_i$ and
switch to the left of the outcome where~$\Delta$ is defined, and this, of
course, contradicts the assumption.
\end{proof}

\subsection{Final verification}
We are now ready to prove the satisfaction of all requirements. The following
two lemmas complete the verification of the construction of
Section~\ref{sec:construction} and the proof of our main theorem.

\begin{lem}[$\cS_i$\nbd requirements]
Every $\cS_i$\nbd requirement is satisfied.
\end{lem}

\begin{proof}
Let~$\beta$ be the last $\cS_i$\nbd parent node along the true path. It is
easy to check that, once we perform diagonalization, then the
$\Psi_i(B;x)$\nbd use is going to be preserved (as we choose the killing
point~$y$ to be the least one such that some~$\gamma(y)$ may enter~$B$ later
in the construction). So we only need to consider the case when we infinitely
often see a believable computation $\Psi_i(B;x)$ but we cannot perform
diagonalization.

Our argument now splits into two cases. One is that there is an $\cS_i$\nbd
child node below~$\beta$ on the true path which has true $g_\alpha$\nbd
outcome (we call this case the $\Sigma_3$\nbd outcome for~$\beta$, i.e., the
requirement is satisfied in a $\Sigma_3$\nbd fashion). The other is that
every $\cS_i$\nbd child node below~$\beta$ on the true path has true $c$\nbd
outcome (similarly, we say~$\beta$ has true $\Pi_3$\nbd outcome).

In the first case, obviously according to the criterion at~$\beta$,
$\psi_i(x)\geq \gamma(x')$ for the killing point~$x'$ at~$\beta$, and the
latter goes to infinity by our construction. So $\Psi_i(B;x)$ diverges and
our requirement is satisfied.

In the second case, by our criterion for going to $c$\nbd
outcomes,~$\psi_i(x)$ is going to be greater than or equal to~$\gamma(z)$ for
arbitrary large~$z$, and this also implies that $\Psi_i(B;x)$ diverges.

In addition, in the second case, it is easy to see that, for each claim
point~$z$, all $\cS_i$\nbd child nodes eventually give up using~$z$ and start
using the next~$z'$ as a killing point (later this will allow us to show that
the ``impact'' of this action on each higher-priority~$\Gamma$ is finite).
\end{proof}

\begin{lem}[$\cR_i$\nbd requirements]
Every $\cR_i$\nbd requirement is satisfied.
\end{lem}

\begin{proof}
We let~$\alpha$ be the last $\cR_i$\nbd node along the true path. Of course,
we only need to consider the case that $W=\Phi(A)$ is total, and so we go to
the $i$\nbd outcome of~$\alpha$ infinitely often, building~$\Gamma$. Now if
there is an $\cS_i$\nbd child node along the true path with true $g$\nbd
outcome associated with~$\alpha$, then by Lemma~\ref{lemma delta}, the
function~$\Delta$ built there is going to correctly compute~$W$, and so the
$\cR_i$\nbd requirement is satisfied.

If there is no such $\cS$\nbd child node along the true path, then we need to
argue that for each fixed~$x$, $\gamma(x)$ only changes finitely often, and
so by Lemma~\ref{lemma gamma},~$\Gamma$ is going to be a functional
computing~$A$ from $B \oplus W$, and our $\Phi$\nbd requirement is also
satisfied.

So fix an~$x$. We can assume that $A(x)=0$ in the end, since otherwise,
after~$x$ enters~$A$, the $\Gamma$\nbd use is going to change for the last
time and then settle down forever. By our construction, if~$W$ changes, we
only increase the $B$\nbd use without changing the $W$\nbd use, and so the
only case in which we may increase the $\Gamma$\nbd use forever is that it
happens infinitely often that some $\cS$\nbd child node below~$\alpha$ has
outcomes associated with~$\alpha$ and puts~$\gamma(y)$ for $y\leq x$ into~$B$
during $B$\nbd stages (where~$y$ is the killing point). By induction
hypothesis, we can assume that $\Gamma(B \oplus W_i;x')$ has settled down for
every $x'<x$. Obviously, only finitely many $\cS$\nbd requirements can
use~$x$ as a killing point. Now by the last paragraph of the proof of the
previous lemma and by our assumption, all such child nodes which use~$x$ as
its killing point will eventually give up using~$x$, and so eventually each
$\Gamma(B \oplus W_i;x)$\nbd use settles down.
\end{proof}

\end{document}